\documentclass[graybox]{svmult}

\usepackage{helvet}         
\usepackage{courier}        
\usepackage{type1cm}        
\usepackage{makeidx}         
\usepackage{graphicx}        
\usepackage{multicol}        
\usepackage[bottom]{footmisc}

\usepackage{bm}
\usepackage{amsmath, amsfonts, amssymb,mathrsfs}
\usepackage{epsf}
\usepackage{listings} 
\usepackage{algorithmic,algorithm}
\usepackage{multirow}
\usepackage{enumerate}
\usepackage{amscd}
\usepackage{slashbox,nicefrac}
\usepackage{mathptmx}

\DeclareMathOperator{\ddiv}{div}

\makeindex             

\begin{document}

\title*{Robust Block Preconditioners for Biot's Model}
\titlerunning{Precondition for Biot's Model} 
\author{James ~H. Adler \and Francisco J. Gaspar \and Xiaozhe Hu \and Carmen Rodrigo \and Ludmil ~T. Zikatanov}
\authorrunning{J.H.~Adler \and F.J.~Gaspar \and X.~Hu \and C.~Rodrigo \and L.T.~Zikatanov}
\institute{J.~H. Adler \at Department of Mathematics, Tufts University, Medford, Massachusetts 02155, USA, \email{James.Adler@tufts.edu}
\and F.~J. Gaspar \at Departamento de Matem\'{a}tica Aplicada, Universidad de Zaragoza, Zaragoza, Spain, \email{fjgaspar@unizar.es}
\and X.~Hu \at Department of Mathematics, Tufts University, Medford, Massachusetts 02155, USA, \email{Xiaozhe.Hu@tufts.edu}
\and C.~Rodrigo \at Departamento de Matem\'{a}tica Aplicada, Universidad de Zaragoza, Zaragoza, Spain, \email{carmenr@unizar.es}
\and L.~T. Zikatanov \at Department of Mathematics, Penn State, University Park, Pennsylvania, 16802, USA, \email{ludmil@psu.edu}
}
%
%
\maketitle


\abstract{In this paper, we design robust and efficient block
  preconditioners for the two-field formulation of Biot's
  consolidation model, where stabilized finite-element discretizations
  are used. The proposed block preconditioners are based on the well-posedness of the discrete linear systems.  Block diagonal (norm-equivalent) and block triangular preconditioners are developed, and we prove that these methods are robust with respect to both physical and discretization parameters. Numerical results are presented to support the theoretical results.
}

\section{Introduction} \label{sec:intro} In this work, we study the
quasi-static Biot's model for soil consolidation.  For linearly elastic, homogeneous, and isotropic porous medium, saturated by an incompressible Newtonian fluid, the consolidation is
  modeled by the following system of partial differential
  equations (see~\cite{Biot1}): 
\begin{eqnarray}
\mbox{\rm equilibrium equation:} & & -{\rm div} \, \sigma' +
\alpha \nabla \, p = {\bm g}, \quad {\rm in} \, \Omega, \label{eq11} \\
\mbox{\rm constitutive equation:} & & \sigma' = 2\mu \varepsilon(\bm{u})  + \lambda\ddiv(\bm{u}) I, \quad {\rm in} \, \Omega,
\label{eq12} \\
\mbox{\rm compatibility condition:} & & \varepsilon({\bm u}) = \frac{1}{2}(\nabla {\bm u} + \nabla
{\bm u}^t), \quad {\rm in} \, \Omega,
\label{eq13} \\
\mbox{\rm Darcy's law:} & & {\bm w} = - K \nabla
p, \quad {\rm in} \, \Omega,
\label{eq14} \\
\mbox{\rm continuity equation:} & & -
\alpha \ {\ddiv} \, \partial_t \bm u  -{\ddiv} \, {\bm w} = f, \quad {\rm in} \, \Omega,
\label{eq15}
\end{eqnarray}
where $\lambda$ and $\mu$ are the Lam\'{e} coefficients, $\alpha$ is the Biot-Willis constant (assumed to be one without loss of generality), $K$ is the hydraulic conductivity (ratio of the permeability of the porous medium to the viscosity of the fluid), $I$ is the identity tensor, ${\bm u}$ is the
displacement vector, $p$ is the pore pressure, $\sigma'$ and $\varepsilon$ are the effective stress and strain tensors for the porous medium, and ${\bm w}$ is the percolation velocity of the fluid relative to the soil. The right-hand-side term, ${\bm g}$, is the density of applied body forces and the source term $f$ represents a forced fluid extraction or injection process. Here, we consider a bounded open subset, $\Omega \subset {\mathbb R}^d,\; d = 2, 3$ with regular boundary $\Gamma$.

Suitable discretizations yield a large-scale linear system of
equations to solve at each time step, which are typically
ill-conditioned and difficult to solve in
practice.  
Thus, iterative solution techniques are usually considered.  For the
coupled poromechanics equations considered here, there are two typical
approaches: fully-coupled or monolithic methods and iterative coupling
methods. Monolithic techniques solve the resulting linear system
simultaneously for all the involved unknowns. In this context,
efficient preconditioners are developed to accelerate the convergence
of Krylov subspace methods and special smoothers are designed in a
multigrid framework. Examples of this approach for poromechanics is
found in \cite{Bergamaschi2007, Ferronato2010, Gaspar2004, Luo2017,
  Gaspar2017sub, Lee.J;Mardal.K;Winther.R2017a, Baerland.T;Lee.J;Mardal.K;Winther.R2017a} and the references therein. Iterative coupling
\cite{Kim.J;Tchelepi.H;Juanes.R;others2009a, Kim_PhD}, in contrast, is a sequential approach in
which either the fluid flow problem or the geomechanics part is solved
first, followed by the solution of the other system. This process is
repeated until a converged solution within a prescribed tolerance is
achieved. The main advantage of iterative coupling methods is that
existing software for simulating fluid flow and geomechanics can be
reused. These type of schemes have been widely studied \cite{Mikelic.A;Wheeler.M2013a,
  Both2017, Almani, Bause}.  In particular, in \cite{Castelleto2015}
and \cite{White.J;Castelletto.N;Tchelepi.H2016a} a re-interpretation of the four commonly used
sequential splitting methods as preconditioned-Richardson iterations
with block-triangular preconditioning is presented. Such analysis
indicates that a fully-implicit method outperforms the convergence
rate of the sequential-implicit methods. Following this idea a family
of preconditioners to accelerate the convergence of Krylov subspace
methods was recently proposed for the three-field formulation of the
poromechanics problem \cite{Castelleto2016}.

In this work, we take the monolithic approach and develop efficient
block preconditoners for Krylov subspace methods for solving the
linear systems of equations arising from the discretization of the
two-field formulation of Biot's model.  These preconditioners take
advantage of the block structure of the discrete problem, decoupling
different fields at the preconditioning stage.
Our theoretical results show their efficiency and robustness 
with respect to the physical and discretization 
parameters.  Moreover, the techniques proposed here 
can also be used for designing fast solvers for the three-field
formulation of the Biot's model.

The paper is organized as follows.  Section~\ref{sec:prelim}
introduces the stabilized finite-element discretizations for the
two-field formulation and the basics of block preconditioners. The
proposed block preconditioners are introduced in
Section~\ref{sec:prec}
.  Finally, in Section~\ref{sec:numerics}, we present numerical
experiments illustrating the effectiveness and robustness of the
proposed preconditioners and make concluding remarks in
Section~\ref{sec:conclusion}.

\section{Two-Field Formulation} \label{sec:prelim}

First, we consider the two-field formulation of Biot's model~\eqref{eq11}-\eqref{eq15}, where the unknowns are the displacement $\bm{u}$ and the pressure $p$.
By considering apropriate Sobolev spaces and integration by parts, we have the following variational form: 
find $\bm{u}(t) \in \bm{H}_0^1(\Omega)$ and $p(t) \in H_0^1(\Omega)$, such that
\begin{eqnarray}
&&a(\bm{u}, \bm{v}) - \alpha (\ddiv \bm{v}, p) = (\bm{g},\bm{v}), \quad \forall  \bm{v}  \in \bm{H}_0^1(\Omega) ,  \label{eqn:weak-2f-1}  \\
&&- \alpha (\ddiv \partial_t \bm{u}, q) - a_p(p,q) = (f, q), \quad \forall q \in H_0^1(\Omega) , \label{eqn:weak-2f-2}
\end{eqnarray}
where
\begin{align*}
a(\bm{u},\bm{v}) &= 2\mu\int_{\Omega}\varepsilon(\bm{u}):\varepsilon(\bm{v}) + \lambda\int_{\Omega} \ddiv\bm{u}\ddiv\bm{v}  \quad \text{and} \quad
a_p(p,q) &= \int_{\Omega} K\nabla p\cdot \nabla q.
\end{align*}
The system is completed with suitable initial data $\bm{u}(0)$ and $p(0)$.

%
%

\subsection{Finite-Element Method} \label{subsec:FEM}
We consider two stable discretizations for the two-field formulation
of Biot's model proposed
in~\cite{Rodrigo.C;Gaspar.F;Hu.X;Zikatanov.L2016a}: 
$\mathbb{P}_1$-$\mathbb{P}_1$ elements and the Mini
element with stabilization.  The fully discretized scheme at time
$t_n$, $n=1,2,\dots$ is as follows: \\
Find $\bm{u}_h^n \in \bm{V_h} \subset \bm{H}^1_0(\Omega)$ and $p_h^n \in Q_h \subset H^1_0(\Omega)$, such that,
\begin{align}
& a(\bm{u}_h^n, \bm{v}_h) - \alpha (\ddiv \bm{v}_h, p_h^n) = (\bm{g}(t_n),\bm{v}_h), \quad \forall \bm{v}_h \in \bm{V}_h, \label{eqn:2f-stab-1}  \\
& -\alpha(\ddiv \bar{\partial}_t \bm{u}_h^n, q_h) - a_p(p_h^n, q_h) - \eta h^2 (\nabla \bar{\partial}_t p^n_h, \nabla q_h) = (f(t_n), q_h), \quad \forall q_h \in Q_h, \label{eqn:2f-stab-2}
\end{align}
where $\bar{\partial}_t \bm{u}_h^n := (\bm{u}_h^n -
\bm{u}_h^{n-1})/\tau$, $\bar{\partial}_t p_h^n := (p_h^n -
p_h^{n-1})/\tau$, and $\eta$ represents the stabilization parameter.  Here, $\bm{V}_h$ and $Q_h$ come from the
$\mathbb{P}_1$-$\mathbb{P}_1$ or Mini element.  At each time step, the linear system has the following two-by-two block form: 
\begin{equation}\label{eqn:2f-linear-stab}
\mathcal{A}\bm{x} = \bm{b}, \quad \mathcal{A}= 
\begin{pmatrix}
A_{\bm{u}} &~& \alpha B^T \\
\alpha B &~& - \tau A_p - \eta h^2 L_p
\end{pmatrix}, 
\
\bm{x} = 
\begin{pmatrix}
\bm{u} \\
p
\end{pmatrix},
\ \text{and} \
\bm{b} = 
\begin{pmatrix}
\bm{f}_{\bm{u}} \\
f_p
\end{pmatrix},
\end{equation}
where $a(\bm{u}, \bm{v}) \rightarrow A_{\bm{u}}$, $-(\ddiv \bm{u},
q) \rightarrow B$,  $a_p( \nabla p, \nabla q) \rightarrow A_p$, and
$(\nabla p, \nabla q) \rightarrow L_p$ represent the discrete versions
of the variational forms.

\subsection{Block Preconditioners} \label{subsec:block-prec}
Next, we introduce the general theory for designing block
preconditioners of Krylov subspace iterative
methods~\cite{Loghin.D;Wathen.A2004, Mardal.K;Winther.R2010}.  Let
$\bm{X}$ be a real, separable Hilbert space equipped with norm
$ \| \cdot \|_{\bm{X}}$ and inner product $(\cdot, \cdot)_{\bm{X}}$.
Also let $\mathcal{A}: \bm{X} \mapsto \bm{X}'$ be a bounded and
symmetric operator induced by a symmetric and bounded bilinear form
$\mathcal{L}(\cdot, \cdot)$, i.e.
$\langle \mathcal{A} \bm{x}, \bm{y} \rangle = \mathcal{L}(\bm{x},
\bm{y})$.
We assume the bilinear form is bounded and satisfies an inf-sup
condition:
\begin{equation}\label{ine:inf-sup-L}
| \mathcal{L}(\bm{x}, \bm{y}) | \leq \beta \| \bm{x} \|_{\bm{X}} \|
\bm{y} \|_{\bm{X}}, \ \forall \bm{x}, \bm{y} \in \bm{X} \quad \mbox{and} \quad \inf_{\bm{x} \in \bm{X}} \sup_{\bm{y} \in \bm{X}} \frac{\mathcal{L}(\bm{x}, \bm{y}) }{\| \bm{x} \|_{\bm{X}} \| \bm{y} \|_{\bm{X}}} \geq \gamma > 0.
\end{equation} 


\subsubsection{Norm-equivalent Preconditioner} \label{subsubsec:norm-equiv}
Consider a symmetric positive definite (SPD) operator  $\mathcal{M}: \bm{X}' \mapsto \bm{X}$ as a
preconditioner for solving~$\mathcal{A} \bm{x} = \bm{b}$.  We define
an inner product $(\bm{x}, \bm{y})_{\mathcal{M}^{-1}}:= \langle
\mathcal{M}^{-1}\bm{x}, \bm{y} \rangle$ on $\bm{X}$ and the
corresponding induced norm is $\| \bm{x} \|^2_{\mathcal{M}^{-1}}:=
(\bm{x}, \bm{x})_{\mathcal{M}^{-1}}$.  It is easy to show that
$\mathcal{MA}: \bm{X} \mapsto \bm{X}$ is symmetric with respect to
$(\cdot, \cdot)_{\mathcal{M}^{-1}}$.  Therefore, we can use
$\mathcal{M}$ as a preconditioner for the MINRES algorithm and use the
following theorem for the convergence rate of preconditioned MINRES.

\begin{theorem}{\cite{Greenbaum.A.1997a}}\label{thm:pminres}
If $\bm{x}^m$ is the $m$-th iteration of MINRES and $\bm{x}$ is the exact solution, then, 
\begin{equation} \label{ine:pminres}
\|\bm{r}^m\|_{\mathcal{M}}
\leq 2 \rho^m \|\bm{r}^0\|_{\mathcal{M}},\quad
\end{equation}
where $\bm{r}^{k} = \mathcal{A}(x-x^k)$ is the residual after the
$k$-th iteration,
$\rho = \frac{\kappa(\mathcal{MA}) - 1}{\kappa(\mathcal{MA}) +1}$, and
$\kappa(\mathcal{MA})$ denotes the condition number of $\mathcal{MA}$.
\end{theorem}

In~\cite{Mardal.K;Winther.R2010}, Mardal and Winther show that, if the well-posedness conditions,~\eqref{ine:inf-sup-L}, hold, and $\mathcal{M}$ satisfies
\begin{equation} \label{def:norm-equiv}
c_1 \| \bm{x} \|_{\bm{X}}^2 \leq \| \bm{x} \|_{\mathcal{M}^{-1}}^2 \leq c_2 \| \bm{x} \|_{\bm{X}}^2,
\end{equation}
then, $\mathcal{A}$ and $\mathcal{M}$ are \emph{norm-equivalent} and
$\kappa(\mathcal{MA}) \leq \frac{c_2 \beta}{c_1 \gamma}$. This implies
that $\rho \leq \frac{c_2 \beta - c_1 \gamma}{ c_2 \beta + c_1
  \gamma}$.  Thus, if the original problem is well-posed and the
constants $c_1$ and $c_2$ are independent of the physical and
discretization parameters, then the convergence rate of preconditioned
MINRES is uniform, hence $\mathcal{M}$ is a robust preconditioner.  

\subsubsection{FOV-equivalent Preconditioner}
\label{subsubsec:FOV-equiv}
In this section we consider the class of
field-of-values-equivalent (FOV-equivalent) preconditioners
$\mathcal{M}_L: \bm{X}' \mapsto \bm{X}$, for GMRES. We define the
notion of FOV-equivalence after the following classical theorem on the
convergence rate of the preconditioned GMRES method.
\begin{theorem}{\cite{Elman.H1982a,Eisenstat.S;Elman.H;Schultz.M1983b}}\label{thm:pgmres}
  If $\bm{x}^m$ is the $m$-th iteration of the GMRES method preconditioned with 
 $\mathcal{M}_L$ and $\bm{x}$
  is the exact solution, then
\begin{equation} \label{ine:pgmres}
\| \mathcal{M}_{L} \mathcal{A}( \bm{x} - \bm{x}^m) \|^2_{\mathcal{M}^{-1}}   \leq \left( 1 - \frac{\Sigma^{2}}{\Upsilon^{2}} \right)^{m}  \| \mathcal{M}_{L} \mathcal{A}( \bm{x}- \bm{x}^{0}) \|^2_{\mathcal{M}^{-1}},
\end{equation}
where, for any $\bm{x} \in \bm{X}$, 
\begin{equation}\label{def:FOV-l}
\Sigma \leq \frac{ (\mathcal{M}_{L} \mathcal{A}  \bm{x} ,\bm{x})_{\mathcal{M}^{-1}} }{(\bm{x}, \bm{x})_{\mathcal{M}^{-1}}}, 
	\quad
	\frac{ \| \mathcal{M}_{L} \mathcal{A} \bm{ x} \|_{\mathcal{M}^{-1}} }{\| \bm{x} \|_{\mathcal{M}^{-1}} } \leq \Upsilon.
\end{equation}
\end{theorem}

If the constants $\Sigma$ and $\Upsilon$ are independent of the
physical and discretization parameters, then $\mathcal{M}_L$ is a
uniform left preconditioner for GMRES and is referred to as
an \emph{FOV-equivalent} preconditioner. In~\cite{Loghin.D;Wathen.A2004}, a block
lower triangular preconditioner has been shown to
satisfy~\eqref{def:FOV-l} based on the well-posedness
conditions,~\eqref{ine:inf-sup-L}, for Stokes/Navier-Stokes equations.
More recently, the same approach has been generalized to Maxwell's
equations~\cite{Adler.J;Hu.X;Zikatanov.L2017a} and
Magnetohydrodynamics~\cite{Ma.Y;Hu.K;Hu.X;Xu.J2016a}.

Similar arguments also apply to right preconditioners
  for GMRES,  $ \mathcal{M}_U: \bm{X}' \mapsto
    \bm{X}$, where 
  the operators, $\mathcal{M}_U$ and $\mathcal{A}$, are FOV equivalent if, for any $\bm{x}' \in \bm{X}'$, 
\begin{equation} \label{def:FOV-r}
\Sigma \leq \frac{ ( \mathcal{A} \mathcal{M}_{U} \bm{x}', \bm{x}')_{\mathcal{M}} }{(\bm{x}', \bm{x}')_{\mathcal{M}} }, 
	\quad
	\frac{ \|  \mathcal{A} \mathcal{M}_{U}  \bm{x'} \|_{\mathcal{M}} }{\| \bm{x}' \|_{\mathcal{M}} } \leq \Upsilon.
\end{equation}
Again, if $\Sigma$ and $\Upsilon$ are independent of the physical and
discretization parameters, $\mathcal{M}_U$ is a uniform right
preconditioner for GMRES.  Such an approach leads to block upper
triangular preconditioners.

\section{Robust Preconditioners for Biot's Model} \label{sec:prec}
In this section, following the framework proposed
in~\cite{Loghin.D;Wathen.A2004, Mardal.K;Winther.R2010} and techniques
recently developed in~\cite{Ma.Y;Hu.K;Hu.X;Xu.J2016a}, we design block
diagonal and triangular preconditioners based on the well-posedness of
the discretized linear system at each time step.  First, we study the
well-posedness of the linear system~\eqref{eqn:2f-linear-stab}.  The
analysis here is similar to the analysis in
~\cite{Rodrigo.C;Gaspar.F;Hu.X;Zikatanov.L2016a}. However, we make
sure that the constants arising from the analysis are independent of
any physical and discretization parameters
. 

The choice of finite-element spaces give $\bm{X} = \bm{V}_h \times
Q_h$, and the finite-element pair satisfies the following inf-sup condition,
\begin{equation}\label{def:weak-inf-sup}
\sup_{\bm{v} \in \bm{V}_h} \frac{(\ddiv \bm{v}, q)}{\| \bm{v} \|_1} \geq \gamma^0_B \| q \| - \xi^0 h \| \nabla q \|, \quad \forall \, q \in Q_h.
\end{equation}
Here, $\gamma_B^0 > 0$ and $\xi^0 \geq 0$ are constants that do not
depend on the mesh size.  Moreover, if we use the Mini-element, $\xi^0
= 0$.  

For $\bm{x} = (\bm{u}, p)^T$, we define the following norm,
\begin{equation} \label{def:norm-stab}
\| \bm{x} \|^2_{\bm{X}} := \| \bm{u} \|^2_{A_{\bm{u}}} + \tau \| p \|_{A_p}^2 + \eta h^2 \| p \|_{L_p}^2 + \frac{\alpha^2}{\zeta^2} \| p \|^2,
\end{equation}
where $\| \bm{u} \|^2_{A_{\bm{u}}} := a(\bm{u}, \bm{u})$, $\| p
\|^2_{A_p} := a_p( \nabla p, \nabla p)$, $\| p \|^2_{L_p} := (\nabla
p, \nabla p)$, $\zeta = \sqrt{\lambda + \frac{2\mu}{d}}$, and $d=2$ or $3$ is
the dimension of the problem.  
With $\zeta$ defined above,
we have $\| \bm v \|_{A_{\bm u}} \leq \sqrt{d} \zeta \| \bm v \|_1$,
and can reformulate the inf-sup condition,~\eqref{def:weak-inf-sup}, as follows,
\begin{equation}\label{ine:weak-inf-sup-mod}
\sup_{\bm v \in \bm V_h} \frac{(B \bm v, q)}{ \| \bm v \|_{A_{\bm u}}} \geq \sup_{\bm v \in \bm V_h} \frac{(B \bm v, q)}{\sqrt{d} \zeta \| \bm v \|_1} \geq \frac{\gamma_B^0}{\sqrt{d} \zeta} \| q \| - \frac{\xi^0}{\sqrt{d} \zeta} h \| \nabla q \| =: \frac{\gamma_B}{\zeta} \| q \| - \frac{\xi}{\zeta} h \| \nabla q \|,
\end{equation}
where $\gamma_B := \gamma_B^0/\sqrt{d}$ and $\xi = \xi^0/\sqrt{d}$.  

Noting that for $d = 2, 3$,
$
2\mu (\varepsilon(\bm v), \varepsilon(\bm v)) \leq a(\bm v, \bm v) \leq (2\mu + d\lambda) (\varepsilon(\bm v), \varepsilon(\bm v))$.
Thus,  $(\ddiv \bm v, \ddiv \bm v) \leq d  (\varepsilon(\bm v), \varepsilon(\bm v))$ and, 
\begin{equation}\label{ine:B-upper}
\zeta^2 \| B \bm v \|^2 = (\lambda + \frac{2 \mu}{d}) \| \ddiv \bm v \|^2 \leq \| \bm v \|^2_{A_{\bm u}} \Longrightarrow 
\| B \bm v \| \leq \frac{1}{\zeta} \| \bm v \|_{A_{\bm u}}.
\end{equation}
This allows us to show that linear system~\eqref{eqn:2f-linear-stab} is well-posed.
\begin{theorem}\label{thm:2f-stab-well-pose}
For $\bm{x} = (\bm{u},p) $ and $\bm{y} = (\bm{v},q)$, let
\begin{equation}\label{def:bilinear-L-stab}
\mathcal{L} (\bm{x}, \bm{y})= (A_{\bm u} \bm{u}, \bm{v}) + \alpha (B \bm{v}, p) + \alpha (B \bm{u}, q) - \tau (K \nabla p, \nabla q) - \eta h^2 (\nabla p, \nabla q).
\end{equation}
Then, \eqref{ine:inf-sup-L} holds and $\mathcal{A}$ defined in~\eqref{eqn:2f-linear-stab} is an isomorphism from $\bm{X}$ to $\bm{X}'$ provided the stabilization parameter $\eta = \delta \frac{\alpha^2 }{\zeta^2}$ with $\delta > 0$.  Moreover, the constants $\gamma$ and $\beta$ are independent of the physical and discretization parameters.
\end{theorem}
\begin{proof}
Based on the inf-sup condition~\eqref{def:weak-inf-sup} and~\eqref{ine:weak-inf-sup-mod}, for any $p$, there exist $\bm{w} \in \bm{V}_h$ such that $(B \bm{w}, p) \geq \left( \frac{\gamma_B}{\zeta} \| p \|  - \frac{\xi}{\zeta} h \| \nabla p \| \right) \|\bm{w} \|_{A_{\bm u}}$ and $\| \bm{w} \|_{A_{\bm u}} = \| p \|$. For given $(\bm{u}, p) \in \bm{V}_h \times Q_h$, we choose $\bm{v} = \bm{u} + \theta  \bm{w}$, $\theta =  \vartheta \frac{\gamma_B \alpha}{\zeta}$ and $q = -p$ and then have,
\begin{align*}
\mathcal{L} (\bm{x}, \bm{y})  &= (A_{\bm u} \bm{u}, \bm{u} + \theta \bm{w}) + \alpha (B(\bm{u} + \theta \bm{w}), p) - \alpha (B \bm{u}, p) \\
&\quad + \tau (K \nabla p, \nabla p) + \eta h^2 (\nabla p, \nabla p) \\
& \geq \| \bm u \|^2_{A_{\bm u}}  - \vartheta \| \bm u \|_{A_{\bm u}} \frac{\gamma_B \alpha}{\zeta} \| p \| + \vartheta \frac{\gamma_B^2 \alpha^2}{\zeta^2} \| p \|^2 - \vartheta \frac{\gamma_B \alpha^2}{\zeta^2} \xi h \| \nabla p \| \| p \|  \\
& \quad + \tau \| p \|^2_{A_p} + \frac{\delta}{\xi^2} \frac{\alpha^2 }{\zeta^2} \xi^2 h^2 \| \nabla p \|^2 \\
& \geq 
\begin{pmatrix}
\| \bm u \|_{A_{\bm u}} \\
\frac{\gamma_B \alpha}{\zeta} \| p \| \\
\frac{\alpha}{\zeta} \xi h \| \nabla p \| \\
\sqrt{\tau} \| p \|_{A_p} 
\end{pmatrix}^T
\begin{pmatrix}
1  &  - \vartheta /2 &  0  &  0 \\
- \vartheta/2 & \vartheta & - \vartheta/2 & 0 \\
0  &  - \vartheta / 2 &  \delta/\xi^2 & 0 \\
0  &  0  &  0   & 1
\end{pmatrix}
\begin{pmatrix}
\| \bm u \|_{A_{\bm u}} \\
\frac{\gamma_B \alpha}{\zeta} \| p \| \\
\frac{\alpha}{\zeta} \xi h \| \nabla p \| \\
\sqrt{\tau} \| p \|_{A_p} 
\end{pmatrix}.
\end{align*}
If $0 < \vartheta < \min\{ 2, \frac{2\delta}{\xi^2} \}$, the matrix in the middle is SPD and there exists $\gamma_0$ such that
\begin{align*}
\mathcal{L} (\bm{x}, \bm{y}) & \geq \gamma_0 \left(  \| \bm u \|^2_{A_{\bm u}} + \frac{\gamma^2_B \alpha^2}{\zeta^2} \| p \|^2  +  \frac{\alpha^2}{\zeta^2} \xi^2 h^2 \| \nabla p \|^2  +  \tau \| p \|^2_{A_p}   \right) \geq \tilde{\gamma}  \| \bm x \|_{\bm X}^2,
\end{align*}
where $\tilde{\gamma} = \gamma_0 \min\{  \gamma_B^2,  \xi^2/\delta \}
$.  Also, it is straightforward to verify $\| (\bm{v}, q) \|_{\bm X}^2
\leq \bar{\gamma}^2 \| (\bm{u}, p) \|_{\bm X}^2$, and the boundedness
of $\mathcal{L}$ by continuity of each term and the Cauchy-Schwarz inequality. Therefore, $\mathcal{L}$ satisfies~\eqref{ine:inf-sup-L} with $\gamma = \tilde{\gamma}/\bar{\gamma}$.
\end{proof} 

\begin{remark}
Note that the choice of $\zeta = \sqrt{\lambda+ 2
    \mu/d}$ is essential to the proof, but is consistent with
  previous practical choice in implementations
  \cite{Aguilar2008,Rodrigo.C;Gaspar.F;Hu.X;Zikatanov.L2016a}.
  Additionally, choosing \emph{any} $\delta > 0$ is sufficient to show
  the well-posedness of the stabilized discretization.  However, for
  eliminating non-physical oscillations of the pressure approximation
  seen in practice~\cite{Aguilar2008}, this is not sufficient, and
  $\delta$ should be sufficiently large. For example, in 1D, we choose
  $\delta = 1/4$.

\end{remark}

\subsection{Block Diagonal Preconditioner} \label{subsec:block-diag}
Now that we have shown \eqref{ine:inf-sup-L} and that the system is
well-posed, we find SPD operators such that~\eqref{def:norm-equiv} is satisfied.  One natural choice is the Reisz operator corresponding to the inner product $(\cdot, \cdot)_{\bm{X}}$,
$
(\mathcal{B} \bm{f}, \bm{x})_{\bm{X}} = \langle \bm{f}, \bm{x} \rangle, \ \forall \bm{f} \in \bm{X}', \ \bm{x} \in \bm{X}. 
$
For the two-field stabilized discretization and the norm $\| \cdot
\|_{\bm{X}}$ defined in~\eqref{def:norm-stab}, we get
\begin{equation} \label{def:2f-diag-prec-stab}
\mathcal{B}_D  = 
\begin{pmatrix}
A_{\bm u} & 0 \\
0  &  \tau A_p + \eta h^2 L_p + \frac{\alpha^2}{\zeta^2} M
\end{pmatrix}^{-1},
\end{equation}
where $M$ is the mass matrix of the pressure block.  Since
$\mathcal{B}_D$ satisfies the norm-equivalent condition with $c_1 =
c_2 = 1$, by Theorem~\ref{thm:2f-stab-well-pose}, we have $\kappa(\mathcal{B} _D \mathcal{A} ) = \mathcal{O}(1)$. 

In practice, applying the preconditioner $\mathcal{B}_D$ involves the action of inverting the diagonal blocks exactly, which is very expensive and infeasible.  Therefore, we replace the diagonal blocks by their spectral equivalent SPD approximations,
\begin{equation*}\label{def:2f-diag-inexact-prec-stab}
\mathcal{M}_D =
\begin{pmatrix}
H_{\bm u} & 0 \\
0 & H_p
\end{pmatrix},
\end{equation*}
where
\begin{align}
& c_{1,\bm u} (H_{\bm u} \bm u, \bm u) \leq (A_{\bm u}^{-1} \bm{u}, \bm{u}) \leq c_{2, \bm u} (H_{\bm u} \bm u, \bm u) \label{ine:Hu-stab} \\
& c_{1,p} (H_{p} p, p) \leq ((\tau A_p +  \eta h^2 L_p+ \frac{\alpha^2}{\zeta^2} M)^{-1} p, p) \leq c_{2, p} (H_{p} p,p). \label{ine:Hp-stab}
\end{align}
Again, $\mathcal{M}_D$ and $\mathcal{A}$ are norm-equivalent and $\kappa(\mathcal{M}_D \mathcal{A} ) = \mathcal{O}(1)$ by Theorem~\ref{thm:2f-stab-well-pose}.

\subsection{Block Triangular Preconditioners} \label{subsec:block-tri}
Next, we consider block triangular preconditioners for the stabilized
scheme, $\mathcal{A}$. For simplicity of the analysis, we modify $\mathcal{A}$ slightly by negating
the second equation.

We consider two kinds of block triangular preconditioners, 
\begin{equation}\label{def:2f-stab-lower-prec}
\mathcal{B}_L = 
\begin{pmatrix}
A_{\bm u} & 0 \\
-\alpha B & \tau A_p + \eta h^2 L_p + \frac{\alpha^2}{\zeta^2} M
\end{pmatrix}^{-1}
\ \text{and} \
\mathcal{M}_L =
\begin{pmatrix}
H_{\bm u}^{-1} & 0 \\
-\alpha B & H_p^{-1}
\end{pmatrix}^{-1},
\end{equation}
and block upper triangular preconditioners,
\begin{equation}\label{def:2f-stab-upper-prec}
\mathcal{B}_U = 
\begin{pmatrix}
A_{\bm u} & \alpha B^T \\
0 & \tau A_p + \eta h^2 L_p + \frac{\alpha^2}{\zeta^2} M
\end{pmatrix}^{-1}
\ \text{and} \
\mathcal{M}_U =
\begin{pmatrix}
H_{\bm u}^{-1} & \alpha B^T \\
0 & H_p^{-1}
\end{pmatrix}^{-1}.
\end{equation}

According to Theorem~\ref{thm:pgmres}, we need to show those block preconditioners satisfies the FOV-equivalence,~\eqref{def:FOV-l} and~\eqref{def:FOV-r}.  We first consider the block lower triangular preconditioner, $\mathcal{B}_L$.  

\begin{theorem} \label{thm:2f-stab-BL}
There exists constants $\Sigma$ and $\Upsilon$, independent of
discretization or physical parameters, such that, for any $\bm{x} = (\bm u, p)^T \neq \bm{0}$, 
\begin{equation*}
\Sigma \leq \frac{(\mathcal{B}_L \mathcal{A} \bm{x}, \bm{x})_{(\mathcal{B}_D)^{-1}}}{(\bm x, \bm x)_{(\mathcal{B}_D)^{-1}}}, \ 
\frac{\| \mathcal{B}_L \mathcal{A} \bm x \|_{(\mathcal{B}_D)^{-1}} }{\| \bm{x} \|_{(\mathcal{B}_D)^{-1}}} \leq \Upsilon,
\end{equation*}
provided that $\eta = \delta
\frac{\alpha^2}{\zeta^2}$ with $\delta > 0$. 
\end{theorem}
\begin{proof}
By direct computation, we have
\begin{align*}
(\mathcal{B}_L \mathcal{A}\bm x, \bm x)_{(\mathcal{B}_D)^{-1}} & =
(\bm u, \bm u)_{A_{\bm u}} + \alpha (B^T p, \bm{u}) + \tau (p,
p)_{A_p}  \\
& \quad + \eta h^2 (L_p p, p) + \alpha^2 (B A_{\bm u}^{-1} B^T p, p) \\
& \geq \Sigma_0 \left(  \| \bm u \|_{A_{\bm u}}^2 + \tau \|p\|_{A_p}^2 + \eta h^2 \| p \|_{L_p}^2+ \alpha^2 \| B^T p \|^2_{A_{\bm u}^{-1}}   \right).
\end{align*}
Note that, due to the inf-sup condition~\eqref{def:weak-inf-sup}, 
\begin{equation*}
 \| B^T p\|_{A_{\bm u}^{-1}} = \sup_{\bm v} \frac{(B \bm v, p)}{\| \bm v \|_{A_{\bm u}}} \geq \frac{\gamma_B}{\zeta} \| p \| - \frac{\xi}{\zeta} h \| \nabla p \|.
\end{equation*}
Therefore, since $\eta = \delta \frac{\alpha^2}{\zeta^2}$ with $\delta > 0$ and by choosing $  \frac{1}{1+ \delta / \xi^2} < \theta < 1$, we have,
\begin{align*}
(\mathcal{B}_L \mathcal{A} \bm x, \bm x)_{(\mathcal{B}_D)^{-1}}  & \geq  \Sigma_0 \left[  \| \bm u \|_{A_{\bm u}}^2 + \tau \|p\|_{A_p}^2 + \eta h^2 \| p \|_{L_p}^2 \right. \\
& \quad \left. + \alpha^2 \left( \frac{\gamma_B}{\zeta} \| p \| - \frac{\xi}{\zeta} h \| \nabla p \|  \right)^2    \right] \\
& \geq  \Sigma_0 \left[  \| \bm u \|_{A_{\bm u}}^2 + \tau \|p\|_{A_p}^2  \right. \\ 
& \quad \left. + (1 - \theta) \frac{\gamma_B^2 \alpha^2}{\zeta^2} \| p \|^2 + \left( 1 + \frac{\delta}{\xi^2} - \frac{1}{\theta}  \right) \frac{\alpha^2}{\zeta^2} \xi^2 h^2 \| \nabla p \|^2  \right] \\
& \geq \Sigma_0 \Sigma_1 \left(  \| \bm u \|_{A_{\bm u}}^2 + \tau \|p\|_{A_p}^2 + \frac{\alpha^2}{\zeta^2} h^2 \| p \|_{L_p}^2+ \frac{\alpha^2}{\zeta^2} \| p \|^2  \right) \\
& =: \Sigma (\bm x, \bm x)_{(\mathcal{B}_D)^{-1}},
\end{align*}
where $\Sigma_1 := \min\{1, (1-\theta) \gamma_B^2,  \left( 1 + \frac{\delta}{\xi^2} - \frac{1}{\theta}  \right) \frac{\xi^2}{\delta}  \} $.  This gives the lower bound.  The upper bound $\Upsilon$ can be obtain directly from the continuity of each term, the Cauchy-Schwarz inequality, and the fact that $\| B^T p\|_{A_{\bm u}^{-1}} \leq \frac{1}{\zeta} \| p \|$ obtained by~\eqref{ine:B-upper}.
 \end{proof}

Similarly, we can show that the other three block preconditioners are
also FOV-equivalent with $\mathcal{A}$ and, therefore, can be used as
preconditioners for GMRES.  Due to the length constraint of
this paper and the fact that the proofs are similar, we only state the results here.
\begin{theorem} \label{thm:2f-stab-ML}
If the conditions~\eqref{ine:Hu-stab} and~\eqref{ine:Hp-stab} hold and
$\| I - H_{\bm u} A_{\bm u} \|_{A_{\bm u}} \leq \rho$ with $0 \leq
\rho <1$, and there exists constants $\Sigma$ and $\Upsilon$, independent of
discretization or physical parameters, such that, for any $\bm{x} = (\bm u, p)^T \neq \bm{0}$, we have
\begin{equation*}
\Sigma \leq \frac{(\mathcal{M}_L \mathcal{A}\bm{x}, \bm{x})_{(\mathcal{M}_D)^{-1}}}{(\bm x, \bm x)_{(\mathcal{M}_D)^{-1}}}, \ 
\frac{\| \mathcal{M}_L \mathcal{A} \bm x \|_{(\mathcal{M}_D)^{-1}} }{\| \bm{x} \|_{(\mathcal{M}_D)^{-1}}} \leq \Upsilon,
\end{equation*}
provided that $\eta = \delta \frac{\alpha^2}{\zeta^2}$ with $\delta > 0$. 
\end{theorem}

\begin{theorem} \label{thm:2f-stab-BU}
There exists constants $\Sigma$ and $\Upsilon$, independent of
discretization or physical parameters, such that, for any $ \bm{0} \neq \bm{x}' \in \bm X' $, we have
\begin{equation*}
\Sigma \leq \frac{( \mathcal{A} \mathcal{B}_U \bm{x}', \bm{x}')_{\mathcal{B}_D}}{(\bm x', \bm x')_{\mathcal{B}_D}}, \ 
\frac{\|  \mathcal{A} \mathcal{B}_U \bm x' \|_{\mathcal{B}_D} }{\| \bm{x}' \|_{\mathcal{B}_D}} \leq \Upsilon,
\end{equation*}
provided that $\eta = \delta \frac{\alpha^2}{\zeta^2}$ with $\delta > 0$.
\end{theorem}

\begin{theorem} \label{thm:2f-stab-MU}
If the conditions~\eqref{ine:Hu-stab} and~\eqref{ine:Hp-stab} hold and $\| I - H_{\bm u} A_{\bm u} \|_{A_{\bm u}} \leq \rho$ with $0 \leq \rho <1$, and there exists constants $\Sigma$ and $\Upsilon$, independent of
discretization or physical parameters, such that, for any $ \bm{0} \neq \bm{x}' \in \bm X' $, we have
\begin{equation*}
\Sigma \leq \frac{( \mathcal{A} \mathcal{M}_U \bm{x}', \bm{x}')_{\mathcal{M}_D}}{(\bm x', \bm x')_{\mathcal{M}_D}}, \ 
\frac{\|  \mathcal{A} \mathcal{M}_U \bm x' \|_{\mathcal{M}_D} }{\| \bm{x}' \|_{\mathcal{M}_D}} \leq \Upsilon,
\end{equation*}
provided that $\eta = \delta \frac{\alpha^2}{\zeta^2}$ with $\delta > 0$. 
\end{theorem}

\begin{remark}
The block upper preconditioner $\mathcal{B}_U$ here is related to the
well-known \emph{fixed-stress split}
scheme~\cite{Kim.J;Tchelepi.H;Juanes.R;others2009a}.  In fact, without
the stabilization term, i.e., $\eta=0$, it is exactly a re-cast of the
fixed-stress split scheme
\cite{White.J;Castelletto.N;Tchelepi.H2016a}.  Moveover, $\zeta^2 =
\lambda + 2\mu/d = : K_{\mathrm{dr}}$, where $K_{\mathrm{dr}}$ is the
drained bulk modulus of the solid. This is exactly the choice
suggested in~\cite{Kim.J;Tchelepi.H;Juanes.R2011a}.  Here, we give a
rigorous theoretical analysis when the fixed-stress split scheme is
used as a preconditioner.  Our analysis is more general in the sense
that $\mathcal{M}_U$ is a inexact version of the fixed-stress split
scheme, and we have generalized it to the finite-element discretization with stabilizations. 
\end{remark}

\section{Numerical Experiments} \label{sec:numerics}
Finally, we provide some preliminary numerical results to demonstrate
the robustness of the proposed preconditioners.  As a discretization, we 
use the stabilized $\mathbb{P}_1$-$\mathbb{P}_1$ scheme 
described in~\cite{Rodrigo.C;Gaspar.F;Hu.X;Zikatanov.L2016a} and 
implemented in the \verb|HAZMATH| library~\cite{Adler.J;Hu.X;Zikatanov.La}.

We consider a 3D footing problem as
in~\cite{Gaspar.F;Gracia.J;Lisbona.F;Oosterlee.C2008}, on the domain,
$\Omega = (-32,32) \times (-32,32) \times (0,64)$.  This is shown in
Figure~\ref{fig:3D-footing}, and represents a block of porous soil.  A
uniform load of intensity $0.1N/m^2$ is applied in a square of size
$32 \times 32 m^2$ at the middle of the top of the domain.  The base of the domain
is assumed to be fixed while the rest of the domain is free to drain.
For the material properties, the Lame coefficients are computed in
terms of the Young modulus, $E$, and the Poisson ratio, $\nu$:
$\lambda = \frac{E \nu}{(1-2\nu)(1+\nu)}$ and $\mu =
\frac{E}{1+2\nu}$.  Since we want to study the robustness of the
preconditioners with respect to the physical parameters, we fix $E = 3
\times 10^4 \, N/m^2$ and let $\nu$ change in the experiments.  The
right side of
Figure~\ref{fig:3D-footing} shows the results of the simulation,
demonstrating the deformation due to a uniform load.

\begin{figure}[htb]
\begin{center}
\caption{Computational domain and boundary conditions} \label{fig:3D-footing}
\includegraphics*[width = 0.3\textwidth]{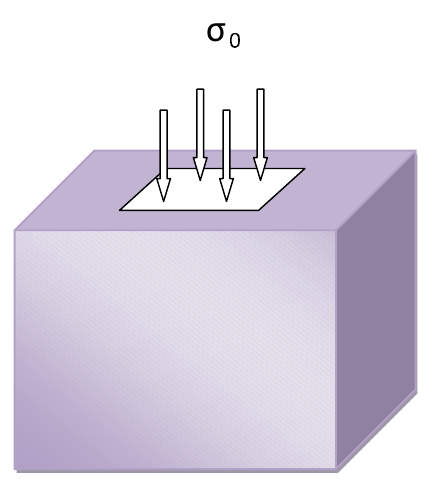}  \quad
\includegraphics*[width = 0.33\textwidth]{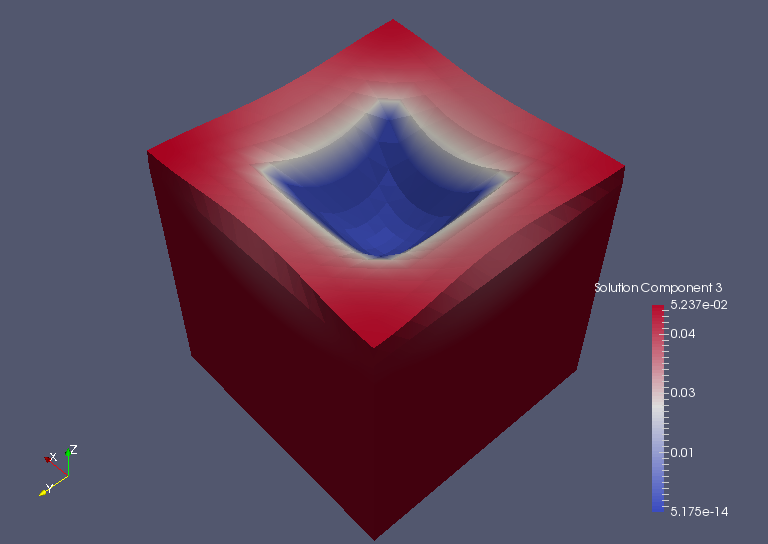} 
\end{center}
\end{figure}

We first study the performance of the preconditioners with respect to
the mesh size $h$ and time step size $\tau$.  Therefore, we fix $K =
10^{-6} \,m^2$ and $\nu = 0.2$.  We use flexible GMRES as the outer iteration with a relative residual stopping criteria of $10^{-6}$.  For $\mathcal{M}_D$, $\mathcal{M}_L$, and $\mathcal{M}_U$, the diagonal blocks are solved inexactly by preconditioned GMRES with a tolerance of $10^{-2}$.  The results are shown in Figure~\ref{tab:block-prec-ht}.  We see that the block preconditioners are effective and robust with respect to the discretization parameters $h$ and $\tau$.

\begin{table}[h!]
\begin{center}
\caption{Iteration counts for the block preconditioners ($*$ means the direct method for solving diagonal blocks is out of memory)}
\begin{tabular}{|l || c c c c|}
		\hline 
&\multicolumn{4}{ |c| }{$\mathcal{B}_D$}\\ \hline 
		 \backslashbox{$\tau$}{$h$} & $\frac{1}{4}$ & $ \frac{1}{8}$ & $\frac{1}{16}$  & $\frac{1}{32}$  \\ 
		\hline 
		$0.1$ & 7 & 7 & 8 & * \\
		$0.01$ & 7 & 7 & 8  & * \\
		$0.001$ & 7 & 7 & 8  & * \\
		$0.0001$ & 7 & 7 & 8 & * \\
		\hline 
		\end{tabular}
\begin{tabular}{|c c c c|}
		\hline
\multicolumn{4}{ |c| }{$\mathcal{B}_L$}\\ \hline
		 $\frac{1}{4}$ \hspace{-38pt}\phantom{\backslashbox{$\tau$}{Mesh}}  & $\frac{1}{8}$ & $\frac{1}{16}$  & $\frac{1}{32}$ 
		  \\ 
		\hline 
		5 & 5 & 6 & * \\
		5 & 5 & 6 & * \\
		5 & 5 & 6 & * \\
		5 & 5 & 6 & * \\
		\hline 
		\end{tabular}
\begin{tabular}{|c c c c|}
		\hline
\multicolumn{4}{ |c| }{$\mathcal{B}_U$}\\ \hline
$\frac{1}{4}$ \hspace{-38pt}\phantom{\backslashbox{$\tau$}{Mesh}}  & $\frac{1}{8}$ & $\frac{1}{16}$  & $\frac{1}{32}$ 		  \\ 
		\hline 
		4 & 4 & 4 & *  \\
		4 & 4 & 5 & * \\
		5 & 5 & 6 & * \\
		5 & 5 & 6 & * \\
		\hline 
		\end{tabular}
\begin{tabular}{|c c c c|}
		\hline
\multicolumn{4}{ |c| }{$\mathcal{M}_D$}\\ \hline
$\frac{1}{4}$ \hspace{-38pt}\phantom{\backslashbox{$\tau$}{Mesh}}  & $\frac{1}{8}$ & $\frac{1}{16}$  & $\frac{1}{32}$ 		  \\ 
		\hline 
		8 & 8 & 9 & 9 \\
		8 & 8 & 9 & 9 \\
		8 & 8 & 9 & 9 \\
		8 & 8  & 9 & 9 \\
		\hline 
		\end{tabular}
\begin{tabular}{|c c c c|}
		\hline
\multicolumn{4}{ |c| }{$\mathcal{M}_L$}\\ \hline
$\frac{1}{4}$ \hspace{-38pt}\phantom{\backslashbox{$\tau$}{Mesh}}  & $\frac{1}{8}$ & $\frac{1}{16}$  & $\frac{1}{32}$ 		  \\ 
		\hline 
		6 & 6 & 8 & 8 \\
		6 & 6 & 8 & 8 \\
		6 & 6 & 8 & 8  \\
		7 & 6 & 8 & 8 \\
		\hline 
		\end{tabular}
\begin{tabular}{|c c c c|}
		\hline
\multicolumn{4}{ |c| }{$\mathcal{M}_U$}\\ \hline
$\frac{1}{4}$ \hspace{-38pt}\phantom{\backslashbox{$\tau$}{Mesh}}  & $\frac{1}{8}$ & $\frac{1}{16}$  & $\frac{1}{32}$ 		  \\ 
		\hline 
		6 & 6 & 8 & 8 \\
		6 & 6 & 8 & 8 \\
		6 & 6 & 8 & 8 \\
		6 & 7 & 8 & 8  \\
		\hline 
		\end{tabular}
\end{center}
\label{tab:block-prec-ht}
\end{table}%

Next, we investigate the robustness of the block preconditioners with respect to the physical parameters $K$ and $\nu$.  We fix the mesh size $h=1/16$ and time step size $\tau = 0.01$.  The results are shown Table~\ref{tab:K-nu}.  From the iteration counts, we can see that the proposed preconditioners are quite robust respect to the physical parameters.

\begin{table}[htp]
\begin{center}
\caption{Iteration counts when varying $K$ or $\nu$} \label{tab:K-nu}
\begin{tabular}{|c|c c c c c c|}
\hline  
 & \multicolumn{6}{ |c| }{$\nu = 0.2$ and varying $K$} \\ \hline
			   & $1$ & $10^{-2}$ & $10^{-4}$ & $10^{-6}$ & $10^{-8}$ & $10^{-10}$  \\ \hline
$\mathcal{B}_D$ & 	    4       &   7  	& 	8 	   & 8   &      8    &    8      \\
$\mathcal{B}_L$ & 	    2       &   5  	& 	6 	   & 6   &      6    &    6      \\
$\mathcal{B}_U$ & 	    3      &    4  	& 	5 	   & 5     &     5    &   5      \\ \hline
$\mathcal{M}_D$ & 	     5      &   8	& 	9	   &  9  &     9       &   9        \\
$\mathcal{M}_L$ & 	     5      &   7	& 	8 	   & 8   &    8        &  8       \\
$\mathcal{M}_U$ &   	     5	     &   7	& 	8	   & 8   &    9        &  8        \\
\hline 
\end{tabular}
\begin{tabular}{|c c c c c c|}
\hline 
 \multicolumn{6}{ |c| }{$K = 10^{-6}$ and varying $\nu$} \\ \hline
$0.1$ & $0.2$ & $0.4$ & $0.45$ & $0.49$ & $0.499$  \\ \hline
7       &   8      &     11    &  11   &    12    &    12      \\
5       &   6      &      8 	&   8    &      8    &    9      \\
4       &   5      &      6 	&   6    &      5    &    4      \\ \hline
8      &    9     &     12	&  13   &     14    &   13        \\
7      &     8     &    11     &  11   &     12     &  12       \\
7      &    8      &     7	&   8     &   17      &  11       \\
\hline 
\end{tabular}
\end{center}
\end{table}%

\section{Conclusions} \label{sec:conclusion}
We have shown that the stability of the discrete problem, using stabilized finite elements, provides the
means for designing robust preconditioners for the
two-field formulation of Biot's consolidation model. Our analysis shows
uniformly bounded condition numbers and uniform convergence rates of the
Krylov subspace methods for the preconditioned linear systems. More
precisely, we prove that the convergence is independent of mesh size,
time step, and the physical parameters of the model.

Current work includes extending this to non-conforming (and
conforming) three-field formulations as in
\cite{Hu.X;Rodrigo.C;Gaspar.F;Zikatanov.L2016a}.  For discretizations
that are stable independent of the physical parameters, uniform
block diagonal preconditioners can be designed using the framework
developed here.  Block lower
and upper triangular preconditioners for GMRES can also be constructed in a
similar fashion.  In addition to their excellent convergence properties, 
the triangular preconditioners naturally provide an
(inexact) fixed-stress split scheme for the three-field formulation.

%
%
%

\bibliographystyle{plain}
\bibliography{Preconditioner4Biot_DD}

\begin{thebibliography}{10}

\bibitem{Adler.J;Hu.X;Zikatanov.La}
James~H. Adler, Xiaozhe Hu, and Ludmil~T. Zikatanov.
\newblock {\verb|HAZMATH|}: A simple finite element, graph, and solver library.

\bibitem{Adler.J;Hu.X;Zikatanov.L2017a}
James~H Adler, Xiaozhe Hu, and Ludmil~T Zikatanov.
\newblock Robust solvers for {M}axwell's equations with dissipative boundary
  conditions.
\newblock {\em SIAM J. Sci. Comput.}, to appear.

\bibitem{Aguilar2008}
G.~Aguilar, F.~Gaspar, F.~Lisbona, and C.~Rodrigo.
\newblock Numerical stabilization of {B}iot's consolidation model by a
  perturbation on the flow equation.
\newblock {\em Internat. J. Numer. Methods Engrg.}, 75(11):1282--1300, 2008.

\bibitem{Almani}
T.~Almani, K.~Kumar, A.~Dogru, G.~Singh, and M.F. Wheeler.
\newblock Convergence analysis of multirate fixed-stress split iterative
  schemes for coupling flow with geomechanics.
\newblock {\em Comput. Methods Appl. Mech. Engrg.}, 311(1):180 -- 207, 2016.

\bibitem{Baerland.T;Lee.J;Mardal.K;Winther.R2017a}
Trygve Baerland, Jeonghun~J Lee, Kent-Andre Mardal, and Ragnar Winther.
\newblock Weakly imposed symmetry and robust preconditioners for biot's
  consolidation model.
\newblock {\em arXiv preprint arXiv:1703.07792}, 2017.

\bibitem{Bause}
M.~Bause, F.A. Radu, and U.~Kocher.
\newblock Space-time finite element approximation of the {B}iot poroelasticity
  system with iterative coupling.
\newblock {\em Computer Methods in Applied Mechanics and Engineering}, 320:745
  -- 768, 2017.

\bibitem{Bergamaschi2007}
L.~Bergamaschi, M.~Ferronato, and G.~Gambolati.
\newblock Block-partitioned solvers for coupled poromechanics: A unified
  framework.
\newblock {\em Comput. Methods Appl. Mech. Engrg.}, 196:2647 -- 2656, 2007.

\bibitem{Biot1}
Maurice~A. Biot.
\newblock General theory of three‐dimensional consolidation.
\newblock {\em Journal of Applied Physics}, 12(2):155--164, 1941.

\bibitem{Both2017}
Jakub~Wiktor Both, Manuel Borregales, Jan~Martin Nordbotten, Kundan Kumar, and
  Florin~Adrian Radu.
\newblock Robust fixed stress splitting for {B}iot's equations in heterogeneous
  media.
\newblock {\em Applied Mathematics Letters}, 68:101 -- 108, 2017.

\bibitem{Castelleto2016}
N.~Castelleto, J.~A. White, and M.~Ferronato.
\newblock Scalable algorithms for three-field mixed finite element coupled
  poromechanics.
\newblock {\em Journal of Computational Physics}, 327:894 -- 918, 2016.

\bibitem{Castelleto2015}
N.~Castelleto, J.~A. White, and H.~A. Tchelepi.
\newblock Accuracy and convergence properties of the fixed-stress iterative
  solution of two-way coupled poromechanics.
\newblock {\em Int. J. Numer. Anal. Meth. Geomech.}, 39:1593 -- 1618, 2015.

\bibitem{Eisenstat.S;Elman.H;Schultz.M1983b}
Stanley~C Eisenstat, Howard~C Elman, and Martin~H Schultz.
\newblock Variational iterative methods for nonsymmetric systems of linear
  equations.
\newblock {\em SIAM Journal on Numerical Analysis}, 20(2):345--357, 1983.

\bibitem{Elman.H1982a}
Howard~C Elman.
\newblock {\em Iterative methods for large, sparse, nonsymmetric systems of
  linear equations}.
\newblock PhD thesis, Yale University New Haven, Conn, 1982.

\bibitem{Ferronato2010}
M.~Ferronato, L.~Bergamaschi, , and G.~Gambolati.
\newblock Performance and robustness of block constraint preconditioners in
  finite element coupled consolidation problems.
\newblock {\em Int. J. Numer. Meth. Engng}, 81:381 -- 402, 2010.

\bibitem{Gaspar.F;Gracia.J;Lisbona.F;Oosterlee.C2008}
F.~J. Gaspar, J.~L. Gracia, F.~J. Lisbona, and C.~W. Oosterlee.
\newblock Distributive smoothers in multigrid for problems with dominating
  grad-div operators.
\newblock {\em Numer. Linear Algebra Appl.}, 15(8):661--683, 2008.

\bibitem{Gaspar2004}
F.~J. Gaspar, F.~J. Lisbona, C.W. Oosterlee, and R.~Wienands.
\newblock A systematic comparison of coupled and distributivesmoothing in
  multigrid for the poroelasticity system.
\newblock {\em Numer Linear Algebra Appl}, 11:93--113, 2004.

\bibitem{Gaspar2017sub}
Francisco~J. Gaspar and Carmen Rodrigo.
\newblock On the fixed-stress split scheme as smoother in multigrid methods for
  coupling flow and geomechanics.
\newblock {\em Submitted}, 2017.

\bibitem{Greenbaum.A.1997a}
A.~Greenbaum.
\newblock {\em {Iterative Methods for Solving Linear Systems}}.
\newblock SIAM, 1997.

\bibitem{Hu.X;Rodrigo.C;Gaspar.F;Zikatanov.L2016a}
Xiaozhe Hu, Carmen Rodrigo, Francisco~J Gaspar, and Ludmil~T Zikatanov.
\newblock A nonconforming finite element method for the {B}iot's consolidation
  model in poroelasticity.
\newblock {\em Journal of Computational and Applied Mathematics}, 310:143--154,
  2017.

\bibitem{Kim_PhD}
J.~Kim.
\newblock {\em Sequential methods for coupled geomechanics and multiphase
  flow}.
\newblock Stanford University, 2010.

\bibitem{Kim.J;Tchelepi.H;Juanes.R2011a}
J~Kim, HA~Tchelepi, and R~Juanes.
\newblock Stability and convergence of sequential methods for coupled flow and
  geomechanics: Fixed-stress and fixed-strain splits.
\newblock {\em Computer Methods in Applied Mechanics and Engineering},
  200(13):1591--1606, 2011.

\bibitem{Kim.J;Tchelepi.H;Juanes.R;others2009a}
Jihoon Kim, Hamdi~A Tchelepi, Ruben Juanes, et~al.
\newblock Stability, accuracy and efficiency of sequential methods for coupled
  flow and geomechanics.
\newblock In {\em SPE reservoir simulation symposium}. Society of Petroleum
  Engineers, 2009.

\bibitem{Lee.J;Mardal.K;Winther.R2017a}
Jeonghun~J Lee, Kent-Andre Mardal, and Ragnar Winther.
\newblock Parameter-robust discretization and preconditioning of biot's
  consolidation model.
\newblock {\em SIAM Journal on Scientific Computing}, 39(1):A1--A24, 2017.

\bibitem{Loghin.D;Wathen.A2004}
D.~Loghin and A.~J. Wathen.
\newblock Analysis of preconditioners for saddle-point problems.
\newblock {\em SIAM J. Sci. Comput.}, 25(6):2029--2049 (electronic), 2004.

\bibitem{Luo2017}
P.~Luo, C.~Rodrigo, F.~J. Gaspar, and C.W. Oosterlee.
\newblock On an {U}zawa smoother in multigrid for poroelasticity equations.
\newblock {\em Numer Linear Algebra Appl}, page e2074.doi:10.1002/nla.2074,
  2017.

\bibitem{Ma.Y;Hu.K;Hu.X;Xu.J2016a}
Yicong Ma, Kaibo Hu, Xiaozhe Hu, and Jinchao Xu.
\newblock Robust preconditioners for incompressible {MHD} models.
\newblock {\em Journal of Computational Physics}, 316:721--746, 2016.

\bibitem{Mardal.K;Winther.R2010}
K.A. Mardal and R.~Winther.
\newblock Preconditioning discretizations of systems of partial differential
  equations.
\newblock {\em Numerical Linear Algebra with Applications}, 2010.

\bibitem{Mikelic.A;Wheeler.M2013a}
Andro Mikeli{\'c} and Mary~F Wheeler.
\newblock Convergence of iterative coupling for coupled flow and geomechanics.
\newblock {\em Computational Geosciences}, 17(3):455--461, 2013.

\bibitem{Rodrigo.C;Gaspar.F;Hu.X;Zikatanov.L2016a}
C~Rodrigo, FJ~Gaspar, Xiaozhe Hu, and LT~Zikatanov.
\newblock Stability and monotonicity for some discretizations of the {B}iot's
  consolidation model.
\newblock {\em Computer Methods in Applied Mechanics and Engineering},
  298:183--204, 2016.

\bibitem{White.J;Castelletto.N;Tchelepi.H2016a}
Joshua~A White, Nicola Castelletto, and Hamdi~A Tchelepi.
\newblock Block-partitioned solvers for coupled poromechanics: A unified
  framework.
\newblock {\em Computer Methods in Applied Mechanics and Engineering},
  303:55--74, 2016.

\end{thebibliography}

\end{document}